\documentclass{amsart}
\usepackage{amssymb}
\usepackage{amscd}
\usepackage[latin1]{inputenc}
\usepackage{amsmath}
\usepackage{amsfonts}
\usepackage{epsfig}
\usepackage{latexsym}
\usepackage{color}
\newtheorem{theorem}{Theorem}[section]
\newtheorem{corollary}[theorem]{Corollary}
\newtheorem{lemma}[theorem]{Lemma}
\newtheorem{definition}[theorem]{Definition}

\newtheorem{proposition}[theorem]{Proposition}

\newtheorem{remark}[theorem]{Remark}

\newtheoremstyle{nonum}{}{}{\upshape}{}{\itshape}{.}{ }{#1#3}
\theoremstyle{nonum}

\newcommand{\comment}[1]{}
\def \beq {\begin{eqnarray}}
\def \eeq {\end{eqnarray}}
\def \beqn {\begin{eqnarray*}}
\def \eeqn {\end{eqnarray*}}

\def\A{{\mathcal A }}

\def\P{{\mathcal P }}
\def\zz{{\mathbb{Z}}}
\def\mZ{{\mathbb{Z}}}
\def\mR{{\mathbb{R}}}
\def\barM{{\overline{M}}}
\def\b0{{\bf{0}}}

\title[Computing bounds for entropy of stationary $\mZ^d$ MRFs]{Computing bounds for entropy of stationary $\mZ^d$ Markov random fields}

\date{}
\author{Brian Marcus}
\address{Brian Marcus\\
Department of Mathematics\\
University of British Columbia}
\email{marcus@math.ubc.ca}
\author{Ronnie Pavlov}
\address{Ronnie Pavlov\\
Department of Mathematics\\
University of Denver}
\email{rpavlov@du.edu}


\begin{document}



\begin{abstract}
For any stationary $\mZ^d$ Gibbs measure that satisfies strong
spatial mixing, we obtain sequences of upper and lower
approximations that converge to its entropy. In the case $d=2$,
these approximations are efficient in the sense that they
are accurate to within $\epsilon$ and can be computed
in time polynomial in $1/\epsilon$.
\end{abstract}

%

\maketitle

\section{Introduction}

The entropy of a stationary $\mZ^d$ Markov random field (MRF) is
notoriously difficult to compute. Recently, Gamarnik and
Katz~\cite{GK} developed a technique for estimating entropy, and
more generally pressure, for certain MRF's.  Their approach built on
earlier work of Weitz~\cite{Weitz} who gave an algorithm for
efficiently counting the number of independent sets in finite
graphs. The algorithm was based on the construction of a computation
tree and the proof of efficiency relied on the concept of strong
spatial mixing (SSM)~\cite[Part 2, Section 2]{Martin}. Coming from
the direction of ergodic theory, we showed that a variant of the
transfer matrix method provides efficient algorithms for estimating
entropy for certain $\mZ^2$ MRF's \cite{pavlov},~\cite{MP}. Our argument
relied on a version of SSM implied by a disagreement percolation
condition developed in~\cite{vdbM} (see Proposition~\ref{vdBM-bound} below).
We regard an algorithm
as ``efficient'' if it computes upper and lower bounds
accurate to within $\epsilon$ in time polynomial in $1/\epsilon$.

While both approaches made use of SSM, they both required other
assumptions as well,
some involving the existence of certain kinds of periodic
configurations.
The purpose of this paper is to give approximations, using only SSM
as a hypothesis, which estimate the entropy of $\mZ^d$ MRF's (and do
so efficiently in the case $d=2$). General sufficient conditions for
SSM can be found in the literature, e.g.,~\cite{dob} and~\cite{vdbM}.

Assuming a standard version of SSM (at exponential rate), we obtain
upper and lower bounds that are exponentially tight (see
Lemma~\ref{xxx} and Theorem~\ref{approx}).
While these bounds are not explicitly computable in all cases,
we believe them to be of independent interest.
In the special case of a nearest-neighbor stationary
$\mZ^d$ Gibbs measure which satisfies SSM, we obtain an algorithm
that approximates these bounds (Theorem~\ref{approx2}). Combining
these results yields an algorithm for approximating entropy that is accurate to
within $\epsilon$ in time polynomial in
$e^{O((\log(1/\epsilon))^{(d-1)^2})}$  (see Corollary~\ref{cor1}).
Specializing to $d=2$, the algorithm runs in time polynomial in
$1/\epsilon$.
We also show how to modify the algorithm to approximate the pressure
of the interaction that defines the Gibbs measure.

%


\section{Background}

We focus on Markov random fields on the
{\bf $d$-dimensional cubic lattice}, the graph 
defined by vertex set $\zz^d $ and edge set $\{\{u,v\} \ :
\sum_{i=1}^d |u_i - v_i| = 1\}$. The {\bf boundary} of a set $S$,
which is denoted by $\partial S$, is the set of $v \in \zz^d
\setminus S$ which are adjacent to some element of $S$. 
%

An {\bf alphabet} $\A$ is a finite set with at least two elements.
For a non-empty subset $S \subset \zz^d$, an element
$u \in \A^S$
is called a {\bf configuration}; here, $S$ is called the {\bf shape}
of $u$. For any configuration $u$ with shape $S$ and any $T
\subseteq S$, denote by $u|_T$ the restriction of $u$ to $T$, i.e.
the sub-configuration of $u$ occupying $T$. For $S,T$ disjoint sets,
$x \in \A^S$ and $y \in \A^T$, $xy$ denotes the configuration on $S
\cup T$ defined by $(xy)|_S = x$ and $(xy)|_T = y$, which we call
the {\bf concatenation} of $x$ and $y$.
We will sometimes informally identify a configuration $x$ on a shape
$S$ with the corresponding configuration on a translate $S+v$,
namely the configuration $y$ on $S + v$ defined by $y_u = x_{u -
v}$.

We use $\sigma$ to denote the {\bf $\mathbb{Z}^d$ shift action} on $\A^{\zz^d}$
defined by $(\sigma_{v}(x))_u = x_{u+v}$. The set $\A^{\zz^d}$ is a
topological space when endowed with the product topology (where $\A$
has the discrete topology), and any subset inherits the induced
topology. By a {\bf $\zz^d$-measure}, we mean a Borel probability
measure on $\A^{\zz^d}$. This means that any $\mu$ is determined by
its values on the sets $[w] := \{x \in \A^{\zz^d} \ : \ x|_S = w\}$,
where $w$ is a configuration with arbitrary finite shape $S
\subseteq \zz^d$. Such sets are called {\bf cylinder sets}, and for
notational convenience, rather than referring to a cylinder set
$[w]$ within a measure or conditional measure, we just use the
configuration $w$. For instance, $\mu(w , v \ | \ u)$ represents the
conditional measure $\mu([w] \cap [v] \ | \ [u])$.  A
$\zz^d$-measure $\mu$ is {\bf translation-invariant} (or {\bf
stationary}) if $\mu(A) = \mu(\sigma_{v} A)$ for all measurable sets
$A$ and $v \in \zz^d$. A $\zz^d$-measure is {\bf fully supported} if
it assigns strictly positive measure to every cylinder set
in $\A^{\zz^d}$.


\begin{definition}
A $\zz^d$-measure $\mu$ is a {\bf $\zz^d$ Markov random field} (or
MRF) if, for any finite $S \subset \zz^d$, any $\eta \in \A^S$, any
finite $T \subset \zz^d$ s.t. $\partial S \subseteq T \subseteq
\zz^d \setminus S$, and any $\delta \in \A^T$ with $\mu(\delta)
> 0$,
\begin{equation}\label{MRFdefn}
\mu(\eta \ | \ \delta|_{\partial S}) = \mu(\eta \ | \ \delta).
\end{equation}
\end{definition}

Informally, $\mu$ is an MRF if, for any finite $S \subset \zz^d$,
configurations on the sites in $S$ and configurations on the sites
in $\zz^d \setminus (S \cup
\partial S)$ are $\mu$-conditionally independent given a configuration on the sites
in $\partial S$. In many papers, the MRF condition is defined in
terms of a parameter $r$, and the set of all sites in $\mathbb{Z}^d \setminus S$ that are
within distance $r$ of $S$ plays the role of $\partial S$. Obviously
our definition corresponds to the case $r=1$
(a ``nearest-neighbor'' MRF).

Another commonly used variant on our definition of MRF involves
conditioning, in the right-hand side of (\ref{MRFdefn}), on an
entire configuration on $\zz^d \setminus S$ a.e. rather than
arbitrarily large finite configurations. However, the definitions
are equivalent (one can just take weak limits) and the finite
approach is a bit more concrete.

For two configurations $y,z \in \A^{T}$ on a finite set $T$, let
$D(y,z) = \{v \in \zz^d:~ y_v \ne z_v\}$. Let $d(\cdot, \cdot)$
denote the $L^1$ distance on $\mZ^d$.

\begin{definition}
A stationary $\mZ^d$ MRF $\mu$ satisfies {\bf strong spatial mixing}
(SSM) if there exist constants $C, \alpha > 0$, such that for any
finite $V \subset \mZ^d$, $u \in V$, $\partial V \subseteq T \subset
V \cup \partial V$, $x \in \A^{\{u\}}$,
and $y, z \in \A^T$   satisfying $\mu(y), \mu(z) > 0$,
$$
\big|\mu(x \ | \ y) - \mu(x \ | \ z)\big| \le Ce^{-\alpha d(\{u\},
D(y,z))}.
$$
\end{definition}

We note that strong spatial mixing can be defined for probability
measures on fairly arbitrary undirected graphs.
Sometimes strong spatial mixing, as we have defined it, is called
``strong spatial mixing with exponential rate.''




The following is the standard notion, in ergodic theory and
information theory, of entropy.

\begin{definition}
Given a $\zz^d$-measure $\mu$ and a finite set $S \subset \zz^d$,
one defines the {\bf entropy} of $\mu$ on $S$ as:
$$
H_\mu(S) = \sum_{w \in \A^S} -  \mu(w) \log(\mu(w))
$$
where terms with $\mu(w) = 0$ are omitted. 
\end{definition}

We also have the notion of conditional entropy.
\begin{definition}\label{condent}
Given a $\zz^d$-measure $\mu$ and disjoint finite sets $S, T \subset
\zz^d$, one defines the {\bf conditional entropy} of $\mu$ on $S$,
given $T$, as:
$$
H_\mu(S \ | \ T) = \sum_{w \in \A^{S\cup T}:~ \mu(w|_T) > 0} -
\mu(w) \log\left(\frac{\mu(w)}{\mu(w|_T)}\right)
$$
where again terms with $\mu(w) = 0$ are omitted. 
\end{definition}

Let $\mu$ be a stationary $\zz^d$-measure. The following
monotonicity property is well known: if $S, T, T' \subset \zz^d$ are
finite, $T' \subset T$ and $S \cap T = \varnothing$, then $H_\mu(S \
| \ T) \le H_\mu(S \ | \ T')$. We can now extend
Definition~\ref{condent} to infinite $T$ by defining
$$
H_\mu(S \ | \ T) = \lim_n H_\mu(S \ | \ T_n)
$$
for a nested sequence of finite sets $T_1 \subset T_2 \subset
\ldots$ with $\cup_n T_n = T$; by the monotonicity property just
mentioned, the limit exists and does not depend on the particular
choice of sequence $T_n$. With this definition, it is clear that the
previously mentioned monotonicity also holds for infinite $T$ and
$T'$:

\begin{lemma}
\label{bounds} Let $\mu$ be a stationary $\zz^d$-measure. If $S, T,
T' \subset \zz^d$, $S$ is finite, $T' \subset T$ and $S \cap T =
\varnothing$, then
\begin{equation*}
 H_\mu(S \ | \ T) \le  H_\mu(S \ | \ T').
\end{equation*}
\end{lemma}

We will find the following notation useful later. Let $S$ and $T$ be
disjoint finite sets. For a stationary $\mZ^d$ MRF $\mu$ and a
fixed configuration $y \in \A^T$, with $\mu(y) > 0$, we define
$$
H_\mu(S \ | \ y) = \sum_{x \in \A^{S}} -
\frac{\mu(xy)}{\mu(y)}\log\left(\frac{\mu(xy)}{\mu(y)}\right).
$$
Thus, we can write
\begin{equation}
\label{decompose1} H_\mu(S \ | \ T) = \sum_{y \in \A^T, ~~ \mu(y) >
0} \mu(y)H_\mu(S \ | \ y).
\end{equation}
If $T$ is the disjoint union of $T_1$ and $T_2$, we can write
\begin{equation}
\label{decompose2} H_\mu(S \ | \ T_1\cup T_2) = \sum_{y \in
\A^{T_1}: ~ \mu(y)
>0} \mu(y) \sum_{w \in \A^{T_2}:~~ \mu(wy) > 0 }
\mu(w \ | \ y) H_\mu(S \ | \ wy).
\end{equation}

We can also define the entropy of a stationary stationary
$\zz^d$-measure itself, also known as entropy rate in information
theory.

\begin{definition} The {\bf measure-theoretic entropy} of a
stationary $\zz^d$-measure $\mu$ on $\A^{\zz^d}$ is defined by
\[
h(\mu)=\lim_{j_1, j_2, \ldots, j_d \rightarrow \infty} \frac{
H_\mu(S_{j_1 \ldots j_d})}{j_1 j_2 \cdots j_d},
\]
where $S_{j_1 j_2 \ldots j_d}$ denotes the $j_1 \times j_2 \times
\ldots \times j_d$ rectangular prism $\prod_{i=1}^d [1, j_i]$.
\end{definition}

It is well known that the limit exists independent of the rates at
which each $j_1, j_2, \ldots, j_d$ approach infinity~\cite[Theorem
15.12]{Georg}.

There is also a useful conditional entropy formula for $h(\mu)$.
For this, we consider the usual lexicographic order on $\zz^d$: $x
\prec y$ if for some $1 \le k \le d$, $x_i = y_i$ for $i=1, \ldots,
k-1$ and $x_k < y_k$. Let $\P^- = \{ z \in \mZ^d: ~ z \prec \b0\}$.
where $\b0$ denotes the origin.

\

\begin{theorem}~\cite[Equation 15.18]{Georg}
Let $\mu$ be a stationary $\zz^d$-measure. Then
$$
h(\mu) = H_\mu(\b0 \ | \ \P^-).
$$
\end{theorem}

\section{Entropy bounds for stationary MRF's}

Let $\P^+ = \{ z \in \mZ^d: ~ z \succeq \b0\}$. Then $\P^+ =
\mathbb{Z}^d \setminus \P^-$. Let $B_n$ denote the $d$-dimensional cube of side
length $2n+1$ centered at $\b0$. Let $S_n = B_n \cap \P^+$, and $U_n
= B_n \cap \partial \P^+$.

We claim that $U_n \subset \partial S_n$.  To see this, note that,
by definition, if $x \in \partial \P^+$, then $x \in \P^-$ and $x$
has a nearest neighbor $y \in \P^+$.  It follows that for some $1
\le k \le d$, we have $x_i = y_i$ for all $i \ne k$ and either ($x_k
= -1$ and $y_k = 0$) or ($x_k = 0$ and $y_k = 1$).  In either case,
if $x \in U_n = B_n \cap \partial \P^+$,  then $y \in B_n$ and so $y
\in S_n$. Thus, $x \in \partial S_n$. Figure~\ref{approxpic1} shows
these sets for $d=2$.

\begin{figure}[h]
\centering
\includegraphics[scale=0.4]{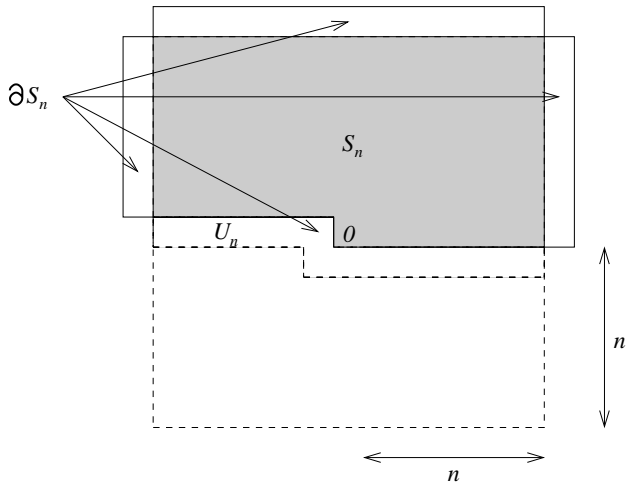}
\caption{$U_n$, $S_n$, and $\partial S_n$} \label{approxpic1}
\end{figure}

\begin{lemma}
\label{xxx} Let $\mu$ be a stationary $\mZ^d$ MRF. Then
\begin{equation}
\label{eqn1} H_\mu(\b0 \ | \ \partial S_n) \le h(\mu) \le H_\mu(\b0
\ | \ U_n).
\end{equation}
\end{lemma}

\begin{proof}
Since $h(\mu) = H_{\mu}(\b0 \ | \ \P^-)$ and $U_n \subset \P^-$, it
follows from Lemma~\ref{bounds} that
\begin{equation}
\label{eqn2} H_\mu(\b0 \ | \ \partial S_n  \cup \P^-) \le  h(\mu)
\le H_\mu(\b0 \ | \ U_n).
\end{equation}
But since $\b0 \in S_n$, $S_n \cap \P^- = \varnothing$ and $\mu$ is
a $\mZ^d$ MRF, it follows that the left-hand sides of (\ref{eqn1})
and (\ref{eqn2}) agree.
\end{proof}

%


We remind the reader of standard notational conventions.  For a
function $f$ on the integers, we write $f= O(n)$ to mean there
exists a constant $C > 0$ such that for sufficiently large $n$,
$|f(n)| \le Cn$ and $f=\Omega(n)$ to mean  there exists a constant
$C > 0$ such that for sufficiently large $n$, $f(n) \ge Cn$.

\begin{theorem}
\label{approx} Let $\mu$ be a stationary $\mZ^d$ MRF that satisfies
SSM. Then
\newline $\big|H_\mu(\b0 \ | \ U_n) - H_\mu(\b0 \ | \ \partial S_n)\big| \ = \
e^{-\Omega(n)}$.
\end{theorem}

\begin{proof}
Let $L_n =  \partial S_n \setminus U_n$.  Then $\partial S_n$ is the
disjoint union of $U_n$ and $L_n$. For every configuration $y \in
\A^{U_n}$ such that $\mu(y)
> 0$, let
$$
E(y) = \{w \in \A^{L_n}: \mu(yw) > 0\}.
$$
By (\ref{decompose1}) and (\ref{decompose2}), we can write
\begin{equation}
\label{compare1} H_\mu(\b0  \ | \ U_n ) = \sum_{y \in \A^{U_n}:
~\mu(y)
> 0 }~ \mu(y) H_\mu(\b0  \ | \ y) \textrm{ and}
\end{equation}
\begin{equation}
\label{compare2} H_\mu(\b0  \ | \ \partial S_n ) = \sum_{y \in
\A^{U_n}: ~\mu(y) > 0}~ \mu(y) \sum_{w \in E(y)}~ \mu(w \ | \
y)H_\mu(\b0  \ | \ yw).
\end{equation}

Fix $y$ as above. Let $C$ and $\alpha$ be the positive constants for
SSM. For any configuration $y$ on $U_n$ and  $w,w' \in E(y)$, we
have
$ d(\{\b0\}, D(w, w')) \ge n $. By SSM applied to $V=S_n$, $T =
\partial S_n$, we have that for all $x \in \A^\b0$, $y \in
\A^{U_n}$, and $w,w' \in E(y)$,
$$
\big|\mu(x \ | \ yw) - \mu(x \ | \ yw')\big| \le Ce^{-\alpha n}.
$$
Now,
$$
\mu(x \ | \ y) = \sum_{w \in E(y)} \mu(w \ | \ y)\mu(x \ | \ yw),
$$
and so for all $w \in E(y)$,
\begin{multline*}
\big|\mu(x \ | \ y) - \mu(x \ | \ yw)\big| = \left|\left(\sum_{w'
\in E(y)} \mu(w' \ | \ y)\mu(x \ | \ yw')\right)
- \mu(x \ | \ yw)\right|\\
= \left|\sum_{w' \in E(y)} \mu(w' \ | \ y)(\mu(x \ | \ yw') -
\mu(x \ | \ yw))\right|\\
\le \sum_{w' \in E(y)} \mu(w' \ | \ y)\big|\mu(x \ | \ yw') - \mu(x
\ | \ yw)\big| \le Ce^{-\alpha n}.
\end{multline*}
Since the function $f(z) = -z\log z$ is H\"{o}lder continuous on
$[0,1]$, it follows that for some $C', \alpha' > 0$,
$$
\big|\mu(x \ | \ y)\log \mu(x \ | \ y) - \mu(x \ | \ yw) \log \mu(x
\ | \ yw) \big| \le C' e^{-\alpha' n}.
$$

%
%
%
%
%
%
%
%
%
%
%
%
%
%
%
%
%
%
%
%
%
%
%
%
%
%
%
%
%
%
%

Thus,
\begin{multline*}
\left|H_\mu(\b0  \ | \ y) - \sum_{w \in E(y)}~ \mu(w \ | \ y)H_\mu(\b0  \ | \ yw)\right| = \left|\sum_{w \in E(y)}~ \mu(w \ | \ y)\big(H_\mu(\b0  \ | \ y) - H_\mu(\b0  \ | \ yw)\big)\right|\\
\le \sum_{w \in E(y)}  \mu(w \ | \ y)\big|H_\mu(\b0  \ | \ y) - H_\mu(\b0  \ | \ yw))\big|\\
\le \sum_{x \in \A^\b0} ~\sum_{w \in E(y)}  \mu(w \ | \ y)\big|\mu(x \ | \ y)\log\mu(x \ | \ y)  - \mu(x \ | \ yw)\log \mu(x \ | \ yw)  \big|\\
\le \sum_{x \in \A^\b0} ~\sum_{w \in E(y)}  \mu(w \ | \ y)
C'e^{-\alpha'n} \le |\A|C'e^{-\alpha' n}.
\end{multline*}
Applying (\ref{compare1}) and (\ref{compare2}), we get
$$
\big|H_\mu(\b0 \ | \ U_n) - H_\mu(\b0 \ | \ \partial S_n)\big| \le
|\A|C'e^{-\alpha' n} = e^{-\Omega(n)}.
$$
\end{proof}

By combining Lemma~\ref{xxx} and Theorem~\ref{approx}, we obtain
exponentially accurate upper and lower approximations to $h(\mu)$
for any stationary $\mZ^d$ MRF $\mu$ which satisfies SSM.  In the following section, we
show that when the MRF is a (nearest-neighbor) Gibbs measure
(defined in the next section) and $d=2$, there is an efficient algorithm to approximate these bounds.
A version of the well-known Hammersley-Clifford
theorem~\cite{preston} shows that any fully supported
(nearest-neighbor) MRF is a (nearest-neighbor) Gibbs measure.
However, that result can fail in general; see~\cite{chandgotia} for
an example based on a construction for finite graphs given
in~\cite{Mou}.




%
%
%

\section{Computation of entropy bounds for stationary Gibbs
measures} \label{sec:compute}

Let $\gamma: \A \rightarrow (0,\infty)$, $\beta_i: \A \times A
\rightarrow [0,\infty)$,  $i = 1, \ldots, d$. For a finite $V
\subset \mZ^d$ and $w \in  \A^V$, let
$$
I(w) = \left(\prod_{v \in V} ~\gamma(v)\right)\prod_{i=1}^d
~~\prod_{\{v \in V: v +e_i \in V\}} ~\beta_i(v, v+e_i).
$$

A configuration  $\delta \in \A^{\partial V}$ is called {\bf
$V$-admissible} if there exists at least one $w \in \A^V$ such that
$I(w\delta)> 0$. 

\begin{definition}
Given $\gamma, \beta_i$ as above, for all $|V| < \infty$ and
$V$-admissible $\delta$, define for all $w \in \A^V$,
$$
\Lambda^{\delta}(w) = \frac{I(w\delta)}{\sum_{x \in \A^V} ~
I(x\delta)}.
$$
The collection $\{\Lambda^{\delta}\}_{V,\delta}$ is called a {\bf
stationary $\mZ^d$ Gibbs specification} for the {\bf local interactions $\gamma$, $\beta_i$}.
\end{definition}

Note that each $\Lambda^{\delta}$ is a probability measure on
$\A^V$, and for $U \subset V$ and $w\in \A^U$,
$$
\Lambda^{\delta}(w) = \sum_{c \in \A^{V\setminus U}}~
\Lambda^{\delta}(wc).
$$
Also, we can regard $\Lambda^{\delta}$ as a probability measure on
configurations $y \in \A^{V \cup \partial V}$ that agree with $\delta$ on
$\partial V$.

%
%
%

\begin{definition}
A {\bf stationary $\mZ^d$ Gibbs measure} for a stationary
$\mZ^d$ Gibbs specification is a stationary $\mZ^d$ MRF $\mu$ on
$\A^{\mZ^d}$ such that for any finite set $V$ and $\delta \in
\A^{\partial V}$, if $\mu(\delta) > 0$ then $\delta$ is
$V$-admissible and for all $x \in \A^V$
$$
\mu(x \ | \ \delta) = \Lambda^{\delta}(x).
$$
\end{definition}

Specifications can be used to define MRF's, not just Gibbs measures
(see~\cite{Georg}).  However, we find the concept of specification most useful for
Gibbs measures.

Gibbs measures, as defined here, are often referred to as
``nearest-neighbor'' Gibbs measures in the literature. Note that
since the $\beta_i$ are allowed to take on the value $0$, a Gibbs
measure need not be fully supported. Also, note that, by definition,
a necessary condition for $\mu(\delta) > 0$ is $V$-admissibility of
$\delta$.  While there may be no finite procedure for determining if
a configuration $\delta$ has positive measure, there is a finite
procedure for determining if $\delta$ is $V$-admissible. For this
reason, we impose an SSM condition on the specification that defines
a Gibbs measure, rather than the Gibbs measure itself.

\begin{definition}
A stationary $\mZ^d$ Gibbs specification $\Lambda$ satisfies {\bf
strong spatial mixing} (SSM) if there exist constants $C, \alpha >
0$, such that for all finite $V \subset \mZ^d$, $u \in V$, $\partial
V \subseteq T \subset V \cup \partial V$, $x \in \A^{\{u\}}$,
$y, z \in \A^T$, such that $\delta = y|_{\partial V}$ and $\eta =
z|_{\partial V}$ are $V$-admissible and $\Lambda^\delta(y),
\Lambda^\eta(z)
> 0$, then
$$
\big|\Lambda^\delta(x \ | \ y) - \Lambda^\eta(x \ | \ z)\big| \le
Ce^{-\alpha d(\{u\}, D(y,z))}.
$$
\end{definition}

Note that if the specification of a Gibbs measure $\mu$ satisfies
SSM, then the measure $\mu$ itself satisfies SSM as an MRF.
It is well known that when the specification satisfies SSM
there is
a unique Gibbs measure corresponding to the specification.
In fact, a weaker notion of spatial mixing, known as weak spatial mixing~\cite{Martin}, is sufficient.


A simple application of the chain rule for
probability distributions shows that our definition of SSM also
implies a version for conditional distributions on larger sets.

\begin{lemma}
\label{ssm-set} For any  stationary $\mZ^d$ Gibbs specification  that
satisfies SSM,  there exist constants $C, \alpha > 0$, such that for
any finite $V \subset \mZ^d$, $U \subseteq V$, $\partial V \subseteq
T \subset V \cup \partial V$, $x \in \A^{U}$,
$y, z \in \A^T$, such that $\delta = y|_{\partial V}$ and $\eta =
z|_{\partial V}$ are $V$-admissible and $\Lambda^\delta(y),
\Lambda^\eta(z)
> 0$, then
\begin{equation}
\label{upper-bound-ssm}
\big|\Lambda^\delta(x \ | \ y) - \Lambda^\eta(x \ | \ z)\big| \le
|U|Ce^{-\alpha d(U, D(y,z))}.
\end{equation}
(The constants $C$, $\alpha$ can be taken to be those in the
definition of SSM.)
\end{lemma}
\begin{proof}
Arbitrarily order the sites in $U$ as
$1, 2, \ldots, {|U|}$.
Then
$$
\big|\Lambda^\delta(x \ | \ y) - \Lambda^\eta(x \ | \ z)\big| =
\left| \left(\prod_{i=1}^{|U|}\Lambda^\delta(x_i \ | \ y, x_1,
\ldots x_{i-1})\right) - \left(\prod_{i=1}^{|U|}\Lambda^\eta(x_i \ |
\ z, x_1, \ldots x_{i-1})\right)\right|
$$
\begin{multline*}
\le \Bigg[ \sum_{i=1}^{|U|}
\left(\prod_{j=1}^{i-1}\Lambda^\delta(x_j \ | \ y, x_1, \ldots
x_{j-1})\right) \left(\prod_{j=i+1}^{|U|}
\Lambda^\eta(x_j \ | \ z, x_1, \ldots x_{j-1})\right)\\
\big|\Lambda^\delta(x_i \ | \ y, x_1, \ldots x_{i-1}) -
\Lambda^\eta(x_i \ | \ z, x_1, \ldots x_{i-1})\big| \Bigg] \le
C|U|e^{-\alpha d(U, D(y,z))}.
\end{multline*}
\end{proof}

The following is the main result of this section.

\begin{theorem}
\label{approx2} Let $\mu$ be a stationary $\mZ^d$ Gibbs measure
whose specification satisfies SSM. Let $(K_n)$, $n \in \mathbb{N}$,
be a sequence of sets satisfying $K_n \subset B_n$ and $|K_n| = O(n^{d-1})$.
Then there is an algorithm which, on input $n$,
computes upper and lower bounds to $H_\mu(\b0 \ | \ K_n)$ in time
$e^{O(n^{(d-1)^2})}$ to within tolerance $e^{-n^{d-1}}$.
\end{theorem}


\begin{remark}
{\rm For this and all subsequent results involving running time of
algorithms involving $\mu$, we do not count computation of the Gibbs
parameters $\gamma$ and $\beta_i$ towards the claimed running time.
(In other words, we assume that we are given approximations to
$\gamma$ and $\beta_i$ with arbitrarily good precision before
performing any computation.)
We also note that the algorithms here do not depend on knowledge of
specific values of the parameters $C$ and $\alpha$ of SSM.}
\end{remark}

As an immediate consequence of Lemma~\ref{xxx},
Theorem~\ref{approx}, and Theorem~\ref{approx2} (applied to $K_n =
\partial S_{n-1}$ and $K_n = U_n$), we have:

\begin{corollary}
\label{cor1} Let $\mu$ be a stationary $\mZ^d$ Gibbs measure whose
specification satisfies SSM. Then there is an algorithm
which, on input $n$, computes
upper and lower bounds to $h(\mu)$ in time $e^{O(n^{(d-1)^2})}$ to
within tolerance $e^{-\Omega(n)}$.
\end{corollary}


Note that for $d=2$ this gives an algorithm to
compute $h(\mu)$ to within $O(1/n)$ in  polynomial time (in $n$). 
\medskip

For the proof of Theorem~\ref{approx2}, we will need the following
result.

%

\begin{lemma}
\label{compute-lemma} Let $\mu$ be a stationary
 $\mZ^d$ Gibbs measure. Let $(K_n)$, $n \in \mathbb{N}$, be a sequence of sets
satisfying $K_n \subset B_n$ and $|K_n| = O(n^{d-1})$. Then for any sequence $(m_n)_{n \in \mathbb{N}}$ of positive
integers, there is an algorithm
which, on input $n$, determines which $\delta \in \A^{\partial B_{n+m_n}}$ are $B_{n+m_n}$-admissible and,
for those which are, computes $\Lambda^\delta(w)$ for all $w\in A^{K_n}$, in running time $e^{O((n+m_n)^{d-1})}$.
\end{lemma}


\begin{proof} For simplicity, we prove this only for $d=2$.  The general case follows
along similar lines.

Fix sequences $(K_n)$ and $(m_n)$, a particular value of $n$, and $w$ and $\delta$ as in the statement of the theorem.
$$
I^\delta(w) := \sum_{c \in \A^{B_{n+m_n}\setminus K_n}} ~I(wc\delta).
$$
We will show that
\begin{equation}
\label{matrix-product} I^\delta(w) = I(\delta)x \left(\prod_{i= -n -
m_n }^{n+m_n-1} \barM_i \right) y,
\end{equation}
where each $\barM_i$ is a square matrix and $x,y$ are vectors, all
indexed by $\A^{[-n-m_n , n+m_n]}$.  For $a \in \A^{[-n-m_n, n+m_n]}$, we
write $a = a_{-n-m_n}, \ldots a_{n+m_n}$.



For $i = -n-m_n, \ldots, n+m_n-1,$  define the transfer matrix
\begin{multline*}
(M_i)_{(a,b)} = \Bigg[\left(\prod_{j=-n-m_n}^{n+m_n} \gamma(a_j)
\beta_1(a_j, b_j) \right)
\left(\prod_{j=-n-m_n}^{n+m_n-1} \beta_2(a_j, a_{j+1})\right)\\
\beta_2(\delta_{i,-n-m_n-1}, a_{-n-m_n}) \beta_2(a_{n+m_n}, \delta_{i, n +
m_n +1})\Bigg].
\end{multline*}
Let $V_i =  \{i\} \times [-n-m_n+1, \ldots n+m_n-1]$ and let
$$
(\barM_i)_{(a,b)} = (M_i)_{(a,b)}
$$
except when $V_i \cap K_n \ne \varnothing$ and $a|_{\{j: (i,j) \in
K_n\}} \ne w|_{V_i \cap K_n}$, in which case we set
$(\barM_i)_{(a,b)} = 0$.   Let
$$
x_a = \prod_{j=-n-m_n}^{n+m_n} \beta_1(\delta_{j,-n-m_n-1}, a_j), \textrm{
and let}
$$
\begin{multline*}
y_a = \left(\prod_{j=-n-m_n}^{n+m_n} \beta_1(a_j,
\delta_{n+m_n+1,j}))\gamma(a_j)\right) \cdot \\
\beta_2(\delta_{n+m_n,-n-m_n-1},
a_{-n-m_n}) \beta_2(a_{n+m_n}, \delta_{n+m_n,n+m_n+1}).
\end{multline*}
The reader can now verify (\ref{matrix-product}).

Note that each $\barM_i$ can be constructed in time
$\big(e^{O(n+m_n)}\big)^2 = e^{O(n+m_n)}$, $x$ and $y$ can be computed
in time $e^{O(n+m_n)}$, and $I(\delta)$ can be computed in time
$O(n+m_n)$.
Each matrix multiplication takes time at most
$\big(e^{O(n+m_n)}\big)^3 = e^{O(n+m_n)}$. Thus, $I^\delta(w)$ can be
computed in time $e^{O(n+m_n)}$. This can be done for all $w \in
\A^{K_n}$ in time $e^{O(n+m_n)}e^{O(n)} = e^{O(n+m_n)}$.

Since
$$
\Lambda^{\delta}(w) = \frac{I^{\delta}(w)}{\sum_{x \in \A^{K_n}}~
I^{\delta}(x)},
$$
we can compute $\Lambda^{\delta}(w)$ for all $w\in \A^{K_n}$ and all
$B_{n+m_n}$-admissible $\delta \in \A^{\partial B_{n+m_n}}$ in time $\big(e^{O(n+m_n)}\big)^2 =
e^{O(n+m_n)}$.

For $d > 2$, the proof follows along similar lines using transfer
matrices indexed by configurations on $(d-1)$-dimensional arrays.


%
%
\end{proof}

\begin{proposition}
\label{compute} Let $\mu$ be a stationary $\mZ^d$ Gibbs measure
whose specification satisfies SSM
with constants $C$ and $\alpha$.
Let $(K_n)$, $n \in \mathbb{N}$,
be a sequence of sets satisfying $K_n \subset B_n$ and $|K_n| = O(n^{d-1})$.
Then for any sequence $(m_n)$ of positive integers,
there is an algorithm which, on input $n$,
computes upper and lower bounds $\mu^+(w)$ and $\mu^-(w)$ to
$\mu(w)$, for all $w \in \A^{K_n}$, in time $e^{O((n+m_n)^{d-1})}$,
such that
$$
\mu^+(w) - \mu^-(w) \le Ce^{-\alpha m_n}|K_n|.
$$
\end{proposition}

\begin{proof}
Fix sequences $(K_n)$ and $(m_n)$, a particular value of $n$, and $w$ as in the statement of the theorem.
Observe that
$$
\mu(w) = \sum_{\delta \in \A^{\partial B_{m_n+n}} :~ \mu(\delta) >
0}~\mu(w \ | \ \delta) \mu(\delta).
$$
Let $\delta^w$ be a configuration $\delta$ which achieves $\max_{\{
B_{n+m_n}\mbox{-admissible}~~   \delta \}}~\Lambda^\delta(w)$ and let
$\delta_w$ be a configuration $\delta$ which achieves  $\min_{
\{B_{n+m_n}\mbox{-admissible} ~~\delta \} } ~\Lambda^\delta(w)$. Since
strict positivity of $\mu(\delta)$ implies $B_{n+m_n}$-admissibility,
it follows that
$$
\Lambda^{\delta_w}(w) \le \mu (w) \le \Lambda^{\delta^w}(w).
$$
Since $\mu$ satisfies SSM, it follows by Lemma~\ref{ssm-set}
(applied to $V=B_{n+m_n}, T=\partial V$ and $U=K_n$)  that
\begin{equation}
\label{diff} 0 \le \Lambda^{\delta^w}(w) - \Lambda^{\delta_w}(w) \le
Ce^{-\alpha m_n}|K_n|.
\end{equation}



By Lemma~\ref{compute-lemma}, we can identify all
$B_{m_n+n}$-admissible $\delta$ and compute $\Lambda^{\delta}(w) $ for
all such $\delta$ and all $w \in \A^{K_n}$ in time
$ e^{O((n +m_n)^{d-1})}$. Thus in time $e^{O((n +m_n)^{d-1})}$
we can identify, for all $w \in \A^{K_n}$, $\delta_w$, and $\delta^w$
and compute the upper and lower bounds $\Lambda^{\delta_w}(w)$ and
$\Lambda^{\delta^w}(w)$.

%
%
This, together with (\ref{diff}), completes the proof.
\end{proof}

Similarly, we have:
\begin{proposition}
\label{compute2} Let $\mu$ be a stationary $\mZ^d$ Gibbs measure
whose specification satisfies SSM
with constants $C$ and $\alpha$.
Let $(K_n)$, $n \in \mathbb{N}$,
be a sequence of sets satisfying $K_n \subset B_n \setminus \{0\}$ and $|K_n| = O(n^{d-1})$.
Then for any sequence $(m_n)$ of positive integers, there is an algorithm which, on input $n$, computes
upper and lower bounds $\mu^+(x_0 \ | \ w)$ and $\mu^-(x_0 \ | \ w)$ to $\mu(x_0 \ | \ w)$
for all $x_0 \in \A$ and $w \in \A^{K_n}$ with $\mu(w) > 0$ in time $e^{O((n+m_n)^{d-1})}$
such that
$$
\mu^+(x_0 \ | \ w) - \mu^-(x_0 \ | \ w) \le Ce^{-\alpha m_n}.
$$
\end{proposition}

\begin{proof}
Fix sequences $(K_n)$ and $(m_n)$, a particular value of $n$, and $w$ as in the statement of the theorem.
Write
$$
\mu(x_0 \ | \ w) = \sum_{\delta \in \A^{\partial B_{m_n+n}}:
~\mu(w\delta)
> 0}~\mu(x_0 \ | \ w, \delta) \mu(\delta \ | \ w).
$$
As in the proof of Proposition~\ref{compute}, we can find
$B_{n+m_n}$-admissible $\delta^{x_0,w}$ and $\delta_{x_0,w}$ such that
$$
\Lambda^{\delta_{x_0,w}}(x_0 | w) \le \mu (x_0 \ | \ w) \le
\Lambda^{\delta^{x_0,w}}(x_0 | w) \textrm{ and}
$$
$$
0 \le \Lambda^{\delta^{x_0,w}}(x_0 | w) -
\Lambda^{\delta_{x_0,w}}(x_0 | w) \le
Ce^{-\alpha m_n}.
$$
(here, we apply SSM to $V = B_{n+m_n},~ T = (\partial V) \cup K_n,~ U
= \{\b0\}$). Then apply Lemma~\ref{compute-lemma} to compute these
bounds, i.e., compute $\Lambda^{\delta^{x_0,w}}(x_0w)$,
$\Lambda^{\delta^{x_0,w}}(w)$, $\Lambda^{\delta_{x_0,w}}(x_0w)$, and
$\Lambda^{\delta_{x_0,w}}(w)$.
\end{proof}

\begin{proof}[Proof of Theorem~\ref{approx2}]
\medskip

Let $(K_n)$, $n \in \mathbb{N}$,
be a sequence of sets satisfying $K_n \subset B_n$ and $|K_n| = O(n^{d-1})$.
We will first describe how to compute upper and lower bounds for arbitrary choice of $(m_n)$, and
then describe how to choose the proper values $(m_n)$ for our algorithm.

For any $n$ and $m_n$, let $\mu^+(w), \mu^-(w), \mu^+(x_0 | w), \mu^-(x_0 | w)$
be as in Propositions~\ref{compute} and~\ref{compute2}.

%

Let $f(x) = -x \log x$.  Let $\mu^{--}(x_0 \ | \ w)$ denote
whichever of $\mu^+(x_0 \ | \ w), \mu^-(x_0 \ | \ w)$ achieves
$\min\big(f(\mu^+(x_0 \ | \ w)), f(\mu^-(x_0 \ | \ w))\big)$.
%
Since $f$ is concave and H\"{o}lder continuous on $[0,1]$,
for some $C', \alpha' > 0$ (independent of $n$ and $m_n$) we have
\begin{equation}
\label{lower-approx} 0 \le f(\mu(x_0 \ | \ w)) - f(\mu^{--}(x_0 \ |
\ w)) \le C'e^{-\alpha' m_n}.
\end{equation}


Recall that
$$
H_\mu(\b0 \ | \ K_n) = \sum_{w \in \A^{K_n}} \mu(w) \sum_{x_0 \in
\A^\b0} f(\mu(x_0 \ | \ w)).
$$
Let $H^-_\mu(\b0 \ | \ K_n)$ denote the expression obtained by
substituting $\mu^{-}(w)$ for $\mu(w)$ and $\mu^{--}(x_0 \ | \ w)$
for $\mu(x_0 \ | \ w)$:
$$
H^-_\mu(\b0 \ | \ K_n) = \sum_{w \in \A^{K_n}} \mu^-(w) \sum_{x_0
\in \A^{\b0}} f(\mu^{--}(x_0 \ | \ w)).
$$
Then $H^-_\mu(\b0 \ | \ K_n) \le H_\mu(\b0 \ | \ K_n)$.

Now, we estimate the difference between $H_\mu(\b0 \ | \ K_n)$ and
$H^-_\mu(\b0 \ | \ K_n)$.
Using (\ref{lower-approx}), we see that
\begin{multline}\label{lowerbd}
H_\mu(\b0 \ | \ K_n) - H^-_\mu(\b0 \ | \ K_n) \\
= \sum_{w \in \A^{K_n}} \mu(w) \sum_{x_0 \in \A} f(\mu(x_0 \ | \ w)) - \sum_{w \in \A^{K_n}} \mu^-(w) \sum_{x_0 \in \A^{\b0}} f(\mu^{--}(x_0 \ | \ w)) \\
= \sum_{w \in \A^{K_n}} \mu(w) \sum_{x_0 \in \A^{\b0}} (f(\mu(x_0 \ | \ w)) - f(\mu^{--}(x_0 \ | \ w)))\\
 + \sum_{w \in \A^{K_n}} (\mu(w) - \mu^-(w)) \sum_{x_0 \in \A^{\b0}} f(\mu^{--}(x_0 \ | \ w))\\
\leq C' e^{-\alpha' m_n} + |\A|^{|K_n|} |K_n| e^{-1} |\A| C
e^{-\alpha m_n} \leq C' e^{-\alpha' m_n} + e^{\eta n^{d-1}} C
e^{-\alpha m_n}
\end{multline}
for some constant $\eta$ (depending on the growth rate of $|K_n|$). The reader can check that there then exists a constant $L$ so that for every $n$, if $m_n > L n^{d-1}$, then $H_\mu(\b0 \ | \ K_n) - H^-_\mu(\b0 \ | \ K_n) < 0.5e^{-n^{d-1}}$.


We also note that the computation time of $H^-_\mu(\b0 \ | \ K_n)$ is $e^{O((n+m_n)^{d-1})}$
(the total amount of time to compute $\mu^{-}(w)$ and
$f(\mu^{--}(x_0 \ | \ w))$ for all $w \in \A^{K_n}$ and $x_0 \in
\A^\b0$.)



For the upper bound, let $\mu^{++}(x_0 \ | \ w)$ be whichever of
$\mu^+(x_0 \ | \ w), \mu^-(x_0 \ | \ w)$ achieves
$\max\big(f(\mu^+(x_0 \ | \ w)), f(\mu^-(x_0 \ | \ w))\big)$ if 
$x,y \le 1/e$ or  $x,y \ge 1/e$, and $1/e$ otherwise. Using
H\"{o}lder continuity of $f$, as well as the fact that
$f(x)$ achieves its maximum at $x = 1/e$, we have:
\begin{equation}
\label{upper-approx} 0 \le f(\mu^{++}(x_0 \ | \ w)) - f(\mu(x_0 \ |
\ w)) \le C'e^{-\alpha' m_n}.
\end{equation}
Then
$$
H^+_\mu(\b0 \ | \ K_n) = \sum_{w \in \A^{K_n}} \mu^+(w) \sum_{x_0
\in \A^{\b0}} f(\mu^{++}(x_0 \ | \ w))
$$
is an upper bound for $H_\mu(\b0 \ | \ K_n)$.

Using (\ref{upper-approx}), we see that
\begin{multline}\label{upperbd}
H^+_\mu(\b0 \ | \ K_n) -  H_\mu(\b0 \ | \ K_n) \\
= \sum_{w \in \A^{K_n}} \mu^+(w) \sum_{x_0 \in \A^{\b0}}  f(\mu^{++}(x_0 \ | \ w)) - \sum_{w \in \A^{K_n}} \mu(w) \sum_{x_0 \in \A^{\b0}} f(\mu(x_0 \ | \ w))\\
= \sum_{w \in \A^{K_n}} \mu(w) \sum_{x_0 \in \A^{\b0}} (f(\mu^{++}(x_0 \ | \ w)) - f(\mu(x_0 \ | \ w))) \\
+ \sum_{w \in \A^{K_n}} (\mu^{+}(w) - \mu(w)) \sum_{x_0 \in \A^{\b0}} f(\mu^{++}(x_0 \ | \ w)) \\
\leq C' e^{-\alpha' m_n} + |\A|^{|K_n|} |K_n| e^{-1} |\A| C
e^{-\alpha m_n} \leq C' e^{-\alpha' m_n} + e^{\eta n^{d-1}} C
e^{-\alpha m_n}.
\end{multline}
For every $n$, if $m_n > Ln^{d-1}$ (the $L$ is the same as for the
lower bound), then $H^+_\mu(\b0 \ | \ K_n) -  H_\mu(\b0 \ | \ K_n) <
0.5e^{-n^{d-1}}$. The time to compute $H^+_\mu(\b0 \ | \ K_n)$ is
$e^{O((n + m_n)^{d-1})}$, the same as for $H^-_\mu(\b0 \ | \ K_n)$.



We now describe the algorithm for choosing the values $(m_n)$. We note that without knowledge of the explicit constants $C$ and $\alpha$ from the strong spatial mixing of $\mu$, we cannot explicitly compute the constant $L$. However, for our purposes, knowledge of $L$ is unnecessary.

The algorithm uses parameters $n$ and $j$ which both start off equal to $1$, though they will be incremented later. The algorithm consists of one main loop which is run repeatedly. At the beginning of the loop, the above bounds $H^-_\mu(\b0 \ | \ K_n)$ and $H^+_\mu(\b0 \ | \ K_n)$ are computed for $m_n = j n^{d-1}$. If the bounds are not within $e^{-n^{d-1}}$ of each other, then $j$ is incremented by $1$ and the algorithm returns to the beginning of the loop. When the bounds are within $e^{-n^{d-1}}$ of each other (which will happen for large enough $j$ by the comments following (\ref{lowerbd}) and (\ref{upperbd})), then $m_n$ is defined to be $jn^{d-1}$, the value of $n$ is incremented by $1$, and the algorithm returns to the beginning of the loop.

By the above discussion, there exists $L$ so that $j$ will never be incremented beyond $L$ in this algorithm. This means that there exists $J$ so that for all sufficiently large $n$, $m_n = J n^{d-1}$. Therefore, for all $n$, the algorithm yields upper and lower bounds to within tolerance $e^{-n^{d-1}}$ in time
$e^{O((n + Jn^{d-1})^{d-1})}$ = $e^{O(n^{(d-1)^2})}$.
 %
%

\end{proof}

\begin{remark}
{\rm We remark that the algorithms in Propositions~\ref{compute}
and~\ref{compute2} can be simplified if one uses knowledge of
specific values of the constants $C$ and $\alpha$ in the definition
of SSM. Namely, one can compute $\Lambda^\delta(w)$ (or
$\Lambda^\delta(x_0 \ | \ w)$) for any fixed $B_{n+m_n}$-admissible
configuration $\delta$ and then set the upper and lower bounds
$\mu^{\pm}(w)$ (or $\mu^{\pm}(x_0 \ | \ w)$) to be
$\Lambda^\delta(w) \pm Ce^{-\alpha m_n}|K_n|$ (or
$\Lambda^{\delta}(x_0 \ | \ w) \pm Ce^{-\alpha m_n}$).


In theory, we can also dispense with the auxiliary sequence $(m_n)$ in Proposition~\ref{compute2}:
we could instead bound $\mu(x_0 \ | \ w)$ by
the minimum and maximum possible values of $\mu(x_0 \ | \ w, \delta)$
for configurations $\delta$ on $\partial B_n$, which would give
approximations of tolerance $Ce^{-\alpha n}$ in time
$O(e^{n^{d-1}})$. A similar simplification could be done for Proposition~\ref{compute}
as well, but it would not be useful for our proof of Theorem~\ref{approx2}: note that in formula
(\ref{lowerbd}), the upper bound on $\mu(w) - \mu^-(w)$ is multiplied by $|A|^{|K_n|}$,
and so this upper bound must be at most $e^{-\Omega(n^{d-1})}$.
Therefore, the described simplification for Proposition~\ref{compute2}
would not reduce the overall order of computation time for Theorem~\ref{approx2}, since
the algorithm from Proposition~\ref{compute} would still require time $e^{O(n^{(d-1)^2})}$.

Finally, we note that in Proposition~\ref{compute2} when $K_n =
\partial S_{n-1}$, there is no need to bound the conditional probabilities $\mu(x_0 \ | \ w)$,
as they can be computed exactly (by using the methods of Lemma~\ref{compute-lemma}).

}
\end{remark}

We will now describe how to extend Theorem~\ref{approx2} and
Corollary~\ref{cor1} to give bounds for pressure in addition to
entropy. Given local interactions $\gamma, \beta_i$, define:
$$
X = \{x \in \A^{\mZ^d}: ~ \mbox{ for all } v
\in \mZ^d \mbox{ and } 1 \le i \le d, ~~\beta_i(x_v, x_{v + e_i}) > 0\}.
$$
$X$ is the set of configurations on $\mZ^d$ defined by
nearest-neighbor constraints and so belongs to the class of
nearest-neighbor (or 1-step) shifts of finite type~\cite{LM}.

Let $f:X \rightarrow \mR$ be
defined by
\begin{equation}
\label{defn_f}
f(x) = \log \gamma(x_0) + \sum_{i=1}^d ~ \log \beta_i(x_0, x_{e_i}).
\end{equation}

\begin{definition}
Let $X$ and $f$ be as above.  Define the {\bf pressure} of $f$ by
$$
P(f) = \max_\nu ~ h(\nu) + \int f d\nu,
$$
where the $\max$ is taken over all stationary measures $\nu$ with support contained in $X$.
A measure which achieves the $\max$  is called an {\bf equilibrium state} for $f$.
\end{definition}

Alternatively, pressure can be defined directly in terms of $X$
and $f$, without reference to stationary measures.
The definition of pressure which we have used is a corollary of the well-known
variational principle~\cite[Chapter 9]{walters}. For general dynamical systems, the $\max$ is
merely a $\sup$; however, in our context, the $\sup$ is always achieved.

It is well known that any equilibrium state for $f$ is a Gibbs measure for
the specification defined by the interactions $\gamma, \beta_i$~\cite[Chapter 4]{Ruelle}. As mentioned earlier,
when the specification satisfies SSM, there is only one Gibbs measure $\mu$ that satisfies
that specification, and so $\mu$ is an (unique) equilibrium state for $f$.

\begin{corollary}
\label{cor2} Let $\gamma, \beta_i$ be local interactions which define
a stationary $\mZ^d$ Gibbs
specification that satisfies SSM. Let $f$ be as in (\ref{defn_f}). Then there is an algorithm to compute
upper and lower bounds to $P(f)$ in time $e^{O(n^{(d-1)^2})}$ to
within tolerance $e^{-\Omega(n)}$.
\end{corollary}

\begin{proof} Let $\mu$ be the unique Gibbs measure that satisfies the specification.
Then Corollary~\ref{cor1} applies to compute such bounds for $h(\mu)$.

It follows from Proposition~\ref{compute} that for any configuration
$w$ on a single site or edge of $\mZ^d$, one can compute upper and
lower bounds to $\mu(w)$ in time $e^{O(n^{d-1})}$ to within
tolerance $e^{-\Omega(n)}$ (in fact, this follows easily from
weak spatial mixing).
Thus one can compute upper and lower bounds to $\int f d\mu$ in the
same time with the same tolerance.

Finally, recall that $\mu$ is an equilibrium state since its specification satisfies SSM,
and so we can compute the desired bounds for $h(\mu) + \int f d\mu = P(f)$.

\end{proof}


There are a variety of conditions in the literature which guarantee
SSM of an MRF: for example, see
\cite{dob},~\cite{GKM},~\cite{GMP},~\cite{RSTVY},~\cite{vdbM}, and
\cite{Weitz}. We present the one from \cite{vdbM} here as one of the
most general and easy to state. Let $\Lambda$ be a stationary Gibbs
specification. Let
$$
q(\Lambda) = \max_{y,z \in \A^{\partial\b0}} d(\Lambda^y,
\Lambda^z),
$$
where $d$ denotes total variation distance of distributions on
$\A^{\b0}$. Let $p_c = p_c(\mZ^d)$ denote the critical probability
for site percolation in $\mZ^d$. (We will not define
$p_c(\mathbb{Z}^d)$ or discuss percolation theory here; for a good
introduction to the subject, see \cite{grimmett}.)

\begin{proposition}
\label{vdBM-bound}
If $q(\Lambda) < p_c$, then $\Lambda$ satisfies SSM.
\end{proposition}

This result is essentially contained in~\cite[Theorem 1]{vdbM};
see~\cite[Theorem 3.10]{MP} for more explanation.
\bigskip

\section*{Acknowledgements} The authors thank David Gamarnik for several
helpful discussions and Andrea Montanari for pointing out the
connection between our
work~\cite{MP},~\cite{pavlov} and
work of Gamarnik and Katz~\cite{GK}.

\bibliographystyle{plain}
\bibliography{MRF_MPv2}

\end{document}